\documentclass[11pt]{amsart}

\usepackage[pdftex]{graphicx}
\usepackage{color}

\usepackage[margin=3.2cm]{geometry}

\usepackage{amsfonts}
\usepackage{amsmath}
\usepackage{amssymb}
\usepackage{amsthm}
\usepackage{subfig}
\usepackage{float}
\usepackage{multirow}

\usepackage{paralist}
\usepackage[colorlinks,cite color=blue,pagebackref=true,pdftex]{hyperref}

\usepackage[T1]{fontenc}

\newcommand\Defn[1]{\textbf{\color{black}#1}}

\renewcommand\emptyset{\varnothing}
\newcommand\Z{\mathbb{Z}}               
\newcommand\Q{\mathbb{Q}}
\newcommand\R{\mathbb{R}}               
               
\newcommand\x{\mathbf{x}}
\newcommand\y{\mathbf{y}}
\newcommand\0{\mathbf{0}}

\newcommand\SymGrp{\mathfrak{S}}

\newcommand\strNum{\delta_G}
\newcommand\SecStrNum{\sigma_G}
\newcommand\Ospace{\mathcal{S}}

\newcommand\parNum{\tilde{\delta}_G}
\newcommand\SecParNum{\tilde{\sigma}_G}

\DeclareMathOperator{\rk}{rank}

\newtheorem{thm}{Theorem}%
\newtheorem{cor}[thm]{Corollary}
\newtheorem{lem}[thm]{Lemma}
\newtheorem{prop}[thm]{Proposition}
\newtheorem{conj}{Conjecture}

\theoremstyle{definition}
\newtheorem{definition}[thm]{Definition}
\newtheorem{example}{Example}
\newtheorem{rem}{Remark}

\newcommand\p{\mathbf{p}}    
    
\newcommand\q{\mathbf{q}}    
\renewcommand\r{\mathbf{r}}    
\newcommand\VarR{\mathrm{V}_{\R}}
\renewcommand\H{\mathcal{H}}

\newcommand\SemiAlg{\mathrm{S}}

\title{Reflection groups, reflection arrangements, and invariant real varieties}

\author{Tobias Friedl}
\address{Fachbereich Mathematik und Informatik, %
Freie Universit\"at Berlin, %
Germany}
\email{tfriedl@zedat.fu-berlin.de}

\author{Cordian Riener}
\address{Aalto Science Institute,
PO Box  11000,
FI-00076 Aalto}
\email{cordian.riener@aalto.fi}

\author{Raman Sanyal}
\address{Institut f\"ur Mathematik, Goethe-Universit\"at Frankfurt, Germany}
\email{sanyal@math.uni-frankfurt.de}


\date{\today}
\thanks{T.~Friedl and R.~Sanyal were supported by the DFG-Collaborative
Research Center, TRR 109 ``Discretization in Geometry and Dynamics''.  
T.~Friedl received additional funding from a scholarship of the Dahlem
Research School at Freie Universit\"at Berlin.}

\begin{document}

\begin{abstract}
    Let $X$ be a nonempty real variety that is invariant under the action of a
    reflection group $G$. We conjecture that if $X$ is defined in terms of the
    first $k$ basic invariants of $G$ (ordered by degree), then $X$ meets a
    $k$-dimensional flat of the associated reflection arrangement.  We prove
    this conjecture for the infinite types, reflection groups of rank at most
    $3$, and $F_4$ and we give computational evidence for $H_4$. This is a
    generalization of Timofte's degree principle to reflection groups.  For
    general reflection groups, we compute nontrivial upper bounds on the
    minimal dimension of flats of the reflection arrangement meeting $X$ from
    the combinatorics of parabolic subgroups. We also give 
    generalizations to real varieties invariant under Lie groups.
\end{abstract}

\maketitle

\section{Introduction}\label{sec:intro}

A real variety $X \subseteq \R^n$ is the set of real points simultaneously
satisfying a system of polynomial equations with real coefficients, that
is, 
\[
    X \ = \ \VarR(f_1,\dots,f_m) \ := \ \{ \p \in \R^n : f_1(\p) = f_2(\p) =
    \cdots = f_m(\p) = 0 \},
\]
for some $f_1,\dots,f_m \in \R[\x] := \R[x_1,\dots,x_n]$.  In contrast to
working over an algebraically closed field, the question if $X \neq \emptyset$
is considerably more difficult to answer, both theoretically and in practice;
see~\cite{roy}. Timofte~\cite{tim} studied real varieties invariant under the
action of the symmetric group $\SymGrp_n$, and proved an
interesting structural result. A $\SymGrp_n$-invariant variety can be defined in terms
of \Defn{symmetric} polynomials, that is, polynomials $f \in \R[\x]$ such that
$f(x_{\tau(1)}, \dots, x_{\tau(n)}) = f(x_1,\dots,x_n)$ for all permutations
$\tau \in \SymGrp_n$. Recall that the fundamental theorem of symmetric
polynomials states that a polynomial $f$ is symmetric if and only if 
$f$ is a polynomial in the  elementary symmetric polynomials $e_1,\dots,e_n$. Let us call a $\SymGrp_n$-invariant variety $X$
\Defn{$\boldsymbol k$-sparse} if $X = \VarR(f_1,\dots,f_m)$ for some symmetric polynomials
$f_1,\dots,f_m \in \R[e_1,\dots,e_k]$.

\begin{thm}[\cite{tim}]\label{thm:Sn-degree}
    Let $X \subseteq \R^n$ be a nonempty $\SymGrp_n$-invariant real variety.  If
    $X$ is $k$-sparse, then there is a point $\p \in X$ with at most $k$
    distinct coordinates.
\end{thm}

Viewing the symmetric group $\SymGrp_n$ as a reflection group in $\R^n$ yields
a sound geometric perspective on this result: As a group of linear
transformations, $\SymGrp_n$ is generated by reflections in the hyperplanes
$H_{ij} = \{\p : p_i = p_j\}$ for $1 \le i < j \le n$. The ambient space
$\R^n$ is stratified by the arrangement of reflection hyperplanes $\H = \{
H_{ij} : i < j \}$.  The closed strata $\H_k \subseteq \R^n $ are the
intersections of $n-k$ linearly independent reflection hyperplanes. Timofte's
result then states that a $k$-sparse variety $X$ is nonempty if and only if $X
\cap \H_k \neq \emptyset$.  Such a point of view can be taken for general real
reflection groups and the aim of this paper is a generalization of
Theorem~\ref{thm:Sn-degree}.

A (real) \Defn{reflection group} $G$ acting on $V \cong \R^n$ is a finite group
of orthogonal transformations generated by reflections. The reflection group
$G$ is \Defn{irreducible} if $G$ is not the product of two nontrivial
reflection groups.  Associated to $G$ is its \Defn{reflection arrangement}
\[
    \H \ = \ \H(G)  \ := \ \{ H = \ker g : g \in G \text{ reflection} \}. 
\] 
The \Defn{flats} of $\H$ are the linear subspaces arising from intersections
of hyperplanes in $\H$. The arrangement of linear hyperplanes stratifies $V$
with strata given by
\[
    \H_i \ = \ \H_i(G) \ := \ \{ \p \in V : \p  \text{ is contained in a flat
    of dimension $i$} \}.
\]
In particular, $\H_n = V$. We call $G$ \Defn{essential} if $G$ does not fix a
nontrivial linear subspace or, equivalently, if $\H_0 = \{0\}$. If $G$ is
essential, then the \Defn{rank} of $G$ is $ \rk(G) := \dim V$.  Reflection
groups naturally occur in connection with Lie groups/algebras and are
well-studied from the perspective of geometry, algebra, and
combinatorics~\cite{hum, FH91, BB05}.  A complete
classification of reflection groups can be given in terms of Dynkin diagrams
(see~\cite{hum}). There are four infinite families of irreducible 
reflection groups $\SymGrp_n \cong A_{n-1},B_n,D_n,I_2(m)$ and six
exceptional reflection groups $H_3, H_4, F_4, E_6, E_7,$ and $E_8$.

The linear action of $G$ on $V$ induces an action on the symmetric algebra
$\R[V]\cong\R[x_1,\ldots,x_n]$ by $g\cdot f(\x):=f(g^{-1}\cdot \x)$.
Chevalley's Theorem~\cite[Ch.~3.5]{hum} states that the ring $\R[V]^G$ of
polynomials invariant under $G$ is generated by algebraically independent
homogeneous polynomials $\pi_1,\pi_2,\dots,\pi_n \in \R[V]$.  The collection
$\pi_1,\dots,\pi_n$ is called a set of \Defn{basic invariants} for $G$. The
basic invariants are not unique, but their degrees $d_i(G) := \deg \pi_i$ are.
Throughout, we will assume that the basic invariants are labelled such that
$d_1 \le d_2 \le \cdots \le d_n$. In accord with $\SymGrp_n$-invariant
varieties, we call a $G$-invariant variety $X = \VarR(f_1,\dots,f_m)$
\Defn{$\boldsymbol k$-sparse} if $f_1,\dots,f_m$ can be chosen in
$\R[\pi_1,\dots,\pi_k]$ for some choice of basic invariants
$\pi_1,\dots,\pi_n$ ordered by nondecreasing degrees. The following is the
main result of the paper.

\begin{thm}\label{thm:main}
    Let $G$ be a reflection group of type $I_2(m), A_{n-1},B_n,D_n, H_3,$ or
    $F_4$ and $X$ a nonempty $G$-invariant real variety. If $X$ is $k$-sparse,
    then $X \cap \H_k(G) \neq \emptyset$.
\end{thm}

Since the first basic invariant of an essential reflection group is a scalar
multiple of $p_2(\x) = \|\x\|^2$, Theorem~\ref{thm:main} is trivially true for
reflection groups of rank $\le 2$. The infinite families $A_{n-1}$, $B_n$, and
$D_n$ are treated in Section~\ref{sec:infinite}.  Timofte's original proof and
its simplification given by the second author in~\cite{rie} use properties of
the symmetric group that are not shared by all reflection groups (such as
$D_n$) and we highlight this difference in Example~\ref{ex:babacon} and
Remark~\ref{rem:babacon}.  In Section~\ref{sec:codim1}, we prove the following
general result that implies Theorem~\ref{thm:main} in the case $k=n-1$. 

\begin{thm}\label{thm:all-but-one}
    Let $G$ be an essential reflection group of rank $n$ and $X =
    \VarR(f_1,\dots,f_m)$ nonempty. If there is $j \in  \{1,\dots,n\}$ such $f_1,\dots,f_m \in
    \R[\pi_i : i \neq j]$, then $X \cap \H_{n-1}(G) \neq \emptyset$. 
\end{thm}

In particular, this result yields Theorem~\ref{thm:main} for all reflection
groups of rank $\le 3$. The group $F_4$ is treated in
Section~\ref{sec:codim1} and we provide computational evidence that
Theorem~\ref{thm:main} also holds for $H_4$. That supports the following
conjecture.

\begin{conj}\label{conj:main}
    Let $G$ be an irreducible and essential reflection group. Then any
    nonempty and $k$-sparse $G$-invariant real variety $X$ intersects $\H_k(G)$.
\end{conj}

Proposition~\ref{prop:conj-equiv} provides a different geometric perspective
on Conjecture~\ref{conj:main} in terms of real orbit spaces and implies
(Proposition~\ref{prop:lower_bound}) that $X$ will in general not meet
$\H_{l}(G)$ for $l < k$.  In Section~\ref{sec:higher_codim}, we prove a weaker
form of Conjecture~\ref{conj:main} under an extra assumption on the defining
polynomials of $X$. In Section~\ref{sec:StrNum}, we obtain upper bounds on the
dimension of the stratum that meets $X$ in terms of the combinatorics of
parabolic subgroups of $G$. Our results generalize to varieties invariant
under the adjoint action of Lie groups and we explore this connection in
Section~\ref{sec:lie}.

Acevedo and Velasco~\cite{velasco} independently considered the related
problem of certifying nonnegativity of $G$-invariant homogeneous polynomials.
They show that low-degree forms (where the exact degree depends on the group)
are nonnegative if and only if they are nonnegative on $\H_{n-1}(G)$.
Questions of nonnegativity of polynomials $f \in \R[V]^G$ are subsumed by our
results. Let us call a $G$-invariant semialgebraic set $S \subseteq V$
$k$-sparse if $S$ is defined in terms of equations and inequalities with
polynomials in $\R[\pi_1,\dots,\pi_k]$.

\begin{prop}\label{prop:nonneg}
    Let $G$ be a reflection group for which
    Conjecture~\ref{conj:main} holds. Let $S \subseteq V$ be a $k$-sparse
    semialgebraic set and let $f \in \R[\pi_1,\dots,\pi_k]$. Then $f$ is
    nonnegative/positive on $S$ if and only if $f$ is nonnegative/positive on
    $\H_k(G) \cap S$.
\end{prop}
\begin{proof}
    If $S$ is $k$-sparse, then the $G$-invariant variety 
    \begin{equation}\label{eqn:X_k}
        X_k(\q) \ := \ \{ \p \in V : \pi_i(\p) = \pi_i(\q) \text{ for } i =
        1,\dots,k \}
    \end{equation}
    is contained in $S$ for any $\q \in S$. Assume that there is a point $\q
    \in S$ with $f(\q) < 0$. By assumption $f = F(\pi_1,\dots,\pi_k)$ for some
    $F \in \R[y_1,\dots,y_k]$. Hence $f$ is negative (and constant) on $X_k(\q)
    \subseteq S$.  By construction $X_k(\q)$ is $k$-sparse and, since $G$
    satisfies Conjecture~\ref{conj:main}, $X_k(\q) \cap \H_k(G) \neq
    \emptyset$.
\end{proof}

The proof of Proposition~\ref{prop:nonneg} makes use of a key observation: It
suffices to consider invariant varieties of the form~\eqref{eqn:X_k} as any
$k$-sparse variety $X$ contains $X_k(\q)$ for all $\q \in X$.  We call
$X_k(\q)$ a \Defn{principal} $k$-sparse variety. Lastly,
let us emphasize again that we will work with \emph{real}
varieties exclusively. In particular, set-theoretically, every real variety $X
= \VarR(f_1,\dots,f_m)$ is the set of solutions to the equation $f(\x) = 0$
for $f = f_1^2 + f_2^2 + \cdots + f_m^2$.

\textbf{Acknowledgements.} We are much indebted to Christian Stump for the
many helpful discussions regarding the combinatorics of reflection groups and
their invariants. We also thank Florian Frick and Christian Haase for an
interesting but fruitless afternoon of orbit spaces. We also thank Mareike
Dressler for help with GloptiPoly and Vic Reiner for suggesting Chevalley's
Restriction Theorem.

\section{The infinite families $A_{n-1},B_n,$ and $D_n$}\label{sec:reflection}\label{sec:infinite}

In this section, we prove Theorem~\ref{thm:main} for the reflection groups of
type $A_{n-1},B_n,$ and $D_n$. The symmetric group $\SymGrp_n$ acts on $\R^n$
but is not essential as it fixes $\R\mathbf{1}$. The restriction to $\{ \x
\in \R^n : x_1 + \cdots + x_n = 0\}$ is the essential reflection group of type
$\mathbf{A_{n-1}}$. The reflection arrangement $\H(\SymGrp_n)$ was described in the
introduction. The $k$-stratum $\H_k(\SymGrp_n)$ is given by the
points $\p \in \R^n$ that have at most $k$ distinct coordinates. A set of basic
invariants is given by the \Defn{elementary symmetric polynomials}
\newcommand\PowSum{s}
\[
    e_k(\x) \ := \ \sum_{1 \le i_1 < \cdots < i_k \le n} x_{i_1}\cdots x_{i_k}
\]
or, alternatively, by the \Defn{power sums} 
\[
    \PowSum_k(\x) \ := \ x_1^k + x_2^k + \cdots + x_n^k,
\]
for $k=1,\dots,n$.  The group $\mathbf{B_n} = \SymGrp_n \rtimes \Z^n_2$ acts
on $V = \R^n$ by \emph{signed} permutations with reflection hyperplanes $\{x_i
= \pm x_j\}$ and $\{ x_i = 0 \}$ for $1 \le i < j \le n$. A point $\p$ lies in
$\H_i(B_n)$ if and only if $(|p_1|,\dots,|p_n|)$ has at most $i$ distinct
nonzero coordinates. A set of basic invariants is given by $\pi_i(\x) =
\PowSum_{2i}(\x) = \PowSum_i(x_1^2,\dots,x_n^2)$.  The index-$2$ subgroup
$\mathbf{D_n}$ of $B_n$ given by the semidirect product of $\SymGrp_n$ with
`even sign changes' yields a reflection group with reflection hyperplanes $\{
x_i = \pm x_j\}$ for $1 \le i < j \le n$. The $k$-stratum of $D_n$ is a bit
more involved to describe: denote by $M$ the set of all $\p \in \R^n$ with
exactly one zero coordinate. Then
\begin{equation}\label{eqn:M}
    \H_k(D_n) \ = \ (\H_k(B_n) \setminus M) \cup (\H_{k-1}(B_n) \cap M).
\end{equation}
The invariant that distinguishes $D_n$ from $B_n$ is given by $e_n(\x) =
x_1x_2\cdots x_n$. A set of basic invariants for $D_n$ are
$\pi_1(\x),\dots,\pi_n(\x)$ with
\begin{equation}\label{eqn:Dn-invs}
    \pi_k(\x) \ := \ 
    \begin{cases}
        \PowSum_{2k}(\x) & \text{ for } 1 \le k \le
        \lfloor\frac{n}{2}\rfloor,\\
        e_n(\x) & \text{ for } k  = 
        \lfloor\frac{n}{2}\rfloor + 1, \text{ and }\\
        \PowSum_{2k-2}(\x) & \text{ for } \lfloor\frac{n}{2}\rfloor + 1 < k
        \le n.\\
    \end{cases}
\end{equation}

We start with the verification of Theorem~\ref{thm:main} for $\SymGrp_n$ which
is exactly Theorem~\ref{thm:Sn-degree}. The proofs for $B_n$ and $D_n$ will
rely on the arguments for $A_{n-1}$.
\newcommand\Jac{\mathrm{Jac}}

\begin{proof}[Proof of Theorem~\ref{thm:main} for $A_{n-1} \cong \SymGrp_n$]
    Let us first assume that $\pi_i(\x) = \PowSum_i(\x) = x_1^i + \cdots +
    x_n^i$ are the power sums for $i=1,\dots,n$. It suffices to show the claim
    for a principal $k$-sparse variety $X_k(\p_0)$ as defined
    in~\eqref{eqn:X_k}, i.e.\ that $X_k(\p_0) \cap \H_k(\SymGrp_n) \neq
    \emptyset$ for $\p_0 \in X$.  Since $\pi_2(\p) = \|\p\|^2$, we conclude
    that $X_k(\p_0)$ is compact and $\pi_{k+1}$ attains its maximum over
    $X_k(\p_0)$ in a point $\q$. At this point, the Jacobian
    $\Jac_\q(\pi_1,\dots,\pi_{k+1}) = \bigl(\nabla \PowSum_1(\q),\ldots,\nabla
    \PowSum_{k+1}(\q)\bigr)$ has rank $< k+1$.  We claim that $\q\in\H_k$ if
    and only if the Jacobian $J = \Jac_\q(\pi_1,\dots,\pi_{k+1})$ has rank $<
    k+1$. Indeed, up to scaling columns, $J$ is given by
    \[
        \begin{pmatrix}
            1   & 1   & \cdots & 1 \\
            q_1 & q_2 & \cdots & q_n \\
            \vdots & & &\vdots \\
            q_1^{k} &q_2^{k}&\cdots & q_n^{k} 
        \end{pmatrix}.
    \]
    The $k+1$ minors are thus Vandermonde determinants all of which vanish if
    and only if $\q \in \H_k$, by the description of $k$-strata for
    $\SymGrp_n$. For an arbitrary choice of basic invariants,
    the result follows from Lemma~\ref{lem:ind} below.
\end{proof}

\begin{lem}\label{lem:ind}
    Fix a reflection group $G$ acting on $V$. Let $\pi_1,\ldots,\pi_n$ and
    $\pi_1',\ldots,\pi_n'$ be two sets of basic invariants and let $1 \le k
    \le n$ such that $d_{k+1}>d_k$.  Then $ \R[\pi_1,\ldots,\pi_k] \ = \
    \R[\pi_1',\ldots,\pi_k'].$ Moreover, $\rk \Jac_\p(\pi_1,\dots,\pi_k) = \rk
    \Jac_\p(\pi_1',\dots,\pi_k')$ for all $\p\in V$.
\end{lem}

\begin{proof}
    For every $1 \le i \le k$, $\pi_i = F_i(\pi'_1,\dots,\pi'_n)$ for some
    polynomial $F_i(y_1,\dots,y_n)$. Homogeneity and algebraic independence
    imply that $F_i \in \R[y_1,\dots,y_k]$. This shows the inclusion
    $\R[\pi_1,\dots,\pi_k] \subseteq \R[\pi_1',\dots,\pi_k']$.  Note that
    \[
        \Jac_{\p}(\pi_1,\dots,\pi_k) \ = \ \Jac_{\pi'(\p)}(F_1,\dots,F_k) \cdot
    \Jac_{\p}(\pi_1',\dots,\pi_k')\]
    for every $\p \in V$.  The same argument
    applied to $\pi'_i$ now proves the first claim and shows that
    $\Jac_{\pi'(\p)}(F_1,\dots,F_k)$ has full rank and this proves the second
    claim.
\end{proof}

We proceed to the reflection groups of type $B_n$.
\begin{proof}[Proof of Theorem~\ref{thm:main} for $B_n$]
    By Lemma~\ref{lem:ind} and the fact that the degrees $d_i(B_n)$ are all
    distinct, we may assume that $\pi_i(\x) = \PowSum_{2i}(\x)$ for
    $i=1,\dots,n$. Moreover, we
    can assume that $X$ is a principal $k$-sparse variety, that is,
    \[
        X \ = \ X_k(\p) \ = \ \{ \x \in \R^n : \PowSum_{2i}(\x) =
        \PowSum_{2i}(\p)
        \text{ for } i = 1,\dots,k \}.
    \]
    Since $X_k(\p) = X_k(\q)$ for all $\q \in X_k(\p)$, we can assume that 
    $\p = (p_1,\dots,p_r,0,\dots,0) \in X$ with the property that $p_1 \cdots
    p_r \neq 0$ and $r$ is minimal. 
    
    If $r = n$, then $X$ does not meet any of the coordinate hyperplanes $\{
    x_i = 0 \}$. Let $\q \in X$ be an extreme point of $\pi_{k+1}$ over $X$.
    At this point, the Jacobian $J = \Jac_\q(\pi_1,\dots,\pi_{k+1})$ does
    not have full rank and hence every maximal minor of
    \[
        J \ = \ 
        \begin{pmatrix}
            q_1 & q_2 & \cdots & q_n \\
            \vdots & & &\vdots \\
            q_1^{2k-1} &q_2^{2k-1}&\cdots & q_n^{2k-1} 
        \end{pmatrix}
    \]
    vanishes. Since $q_i \neq 0$ for all $i=1,\dots,n$, the Vandermonde
    formula implies that $(q_1^2,q_2^2,\dots,q_n^2)$ has at most $k$ distinct
    coordinates, which yields the claim.

    If $r < n$, we can restrict $X$ to the linear subspace $U = \{ \x \in
    \R^n : x_{r+1} = \cdots = x_n = 0\} \cong \R^r$. The set $X' := X \cap U
    \subseteq \R^r$
    is nonempty and, in particular, a $k$-sparse $B_r$-invariant variety that
    stays away from the coordinate hyperplanes in $\R^r$. By the previous
    case, there is a point $\q' \in X'$ such that
    $(|q'_1|,\dots,|q'_r|)$ has at most $k$ distinct coordinates. By
    construction, $\q = (\q',\mathbf{0}) \in X \cap \H_k(B_n)$, which proves the claim.
\end{proof}

The key to the proof of Theorem~\ref{thm:main} for $A_{n-1}$ and $B_n$ is the
strong connection between the strata $\H_k$ and the ranks of the Jacobians
$\Jac(\pi_1,\dots,\pi_{k+1})$.

\begin{cor}\label{cor:strongSteinberg}
    Let $G \in \{ \SymGrp_n, B_n\}$ and $\pi_1,\dots,\pi_n$ a set of basic
    invariants for $G$. Then a point $\q \in V$
    lies in $\H_k(G)$ for $0 \le k \le n-1$ if and only if
    $\Jac_{\p}(\pi_1,\dots,\pi_{k+1})$ has rank at most $k$.
\end{cor}

It is tempting to believe that such a statement holds true for all reflection
groups and, indeed, necessity follows from a well-known result of
Steinberg~\cite{stein}. However, the following example shows that
Corollary~\ref{cor:strongSteinberg} does not hold in general.

\begin{example}\label{ex:babacon}
    Consider the group $G = D_5$ acting on $\R^5$ and the point $\p =
    (1,1,1,1,0)$.  The point lies in
    $\H_2(D_5) \setminus \H_1(D_5)$, that is, $\p$ lies on exactly $3$ linearly
    independent reflection hyperplanes.  On the other hand, for any choice of
    basic invariants $\pi_1,\dots,\pi_5$ the gradients $\nabla_\p \pi_1,
    \nabla_\p \pi_2$ are linearly dependent.  Indeed, for $\pi_1 = \|x\|^2 =
    x_1^2 + \cdots + x_5^2$ and $\pi_2 = x_1^4 + \cdots + x_5^4$, this is easy
    to check and this extends to all choices of basic invariants using
    Lemma~\ref{lem:ind}.
\end{example}

\begin{rem}\label{rem:babacon}
    Example~\ref{ex:babacon} also serves as a counterexample to
    generalizations of Corollary~\ref{cor:strongSteinberg} to all finite
    reflection groups considered in~\cite[Lemma~1']{giv} (without a proof)
    and~\cite[Statement~3.3]{barbancon}.  Moreover, in the language of Acevedo
    and Velasco~\cite[Definition 7]{velasco}, it is the first example of a
    reflection group not satisfying the \emph{minor factorization condition}.  
\end{rem}

The following proof of Theorem~\ref{thm:main} for type $D_n$ does not rely on
an extension of Corollary~\ref{cor:strongSteinberg}.

\begin{proof}[Proof of Theorem~\ref{thm:main} for $D_n$]
    Let $\pi_1,\dots,\pi_n$ be a choice of basic invariants for $D_n$ and let
    $X = X_k(\q) \subseteq \R^n$ for some $\q \in \R^n$ and $1 \le k < n$.  If
    $n$ is odd or if $k \neq  \lfloor \frac{n}{2} \rfloor$, then, by
    Lemma~\ref{lem:ind}, we can assume that the basic invariants are 
    given by~\eqref{eqn:Dn-invs}.  If $n$ is even and $k = \lfloor \frac{n}{2}
    \rfloor$, then $\pi_{\lfloor \frac{n}{2} \rfloor}(\x) \ = \ \alpha
    \PowSum_{n}(\x) + \beta e_n(\x), $ for some $\beta \neq 0$.  We can also
    assume that $\q = (q_1,\dots,q_l,0,\dots,0)$ with $q_1\cdots q_l \neq 0$
    and $l$ maximal among all points in $X_k(\q)$. We distinguish two cases.

    \textbf{Case $l < n$:} In this case, $e_{n}(\x)$ is identically zero on
    $X_k(\q)$ and $X' := X_k(\q) \cap \{ \x : x_n = 0 \}$ is nonempty.  If $k
    \ge \lfloor \frac{n}{2} \rfloor + 1$, then we can identify 
    \[ 
        X' \ = \ \{ \x' \in \R^{n-1} : \PowSum_{2i}(\x',\0)= \PowSum_{2i}(\q)
        \text{ for } i = 1,\dots,k-1 \} .
    \] 
    Hence $X'$ is a real variety in $\R^{n-1}$ invariant under the action of
    $B_{n-1}$ and $X'$ is $(k-1)$-sparse. By
    Theorem~\ref{thm:main} for $B_{n-1}$, $X' \cap \H_{k-1}(B_{n-1}) \neq
    \emptyset$. The claim now follows the description of $\H_k(D_n)$ given 
    in~\eqref{eqn:M}.

    If $k < \frac{n}{2}$, consider the Jacobian of $\pi_1 = \PowSum_{2},\dots,
    \pi_k = \PowSum_{2k}$ and the $(l+1)$-th elementary symmetric polynomial
    $e_{l+1}(\x)$ at $\q$
    \begin{equation}\label{eqn:l_n-1}
        J \ = \ \Jac_\q(\PowSum_2,\dots,\PowSum_{2k},e_{l+1}) \ = \ 
        \begin{pmatrix}
            q_1   & q_2   & \cdots & q_l   & 0 & 0 & \cdots & 0 \\
            q_1^3 & q_2^3 & \cdots & q_l^3 & 0 & 0 & \cdots & 0  \\
            \vdots & \vdots & & \vdots & \vdots & \vdots & \vdots & \vdots \\
            q_1^{2k-1} & q_2^{2k-1} & \cdots & q_l^{2k-1} & 0  & 0 &
            \cdots & 0  \\
            0 & 0 & 0 & 0& q_1\cdots q_l & 0 & \cdots & 0
        \end{pmatrix}.
    \end{equation}
    We observe that the $(l+1)$-th elementary symmetric
    function $e_{l+1}(\x)$ is identically zero on $X_k(\q)$ and hence the gradients of 
    $\pi_1,\dots,\pi_k$ and $e_{l+1}$ are linearly dependent on $X_k(\q)$. In
    particular, the Jacobian $J$ has rank $\le k$. Since 
    $q_1 \cdots q_l \neq 0$, the Vandermonde minors imply
    \[
        \prod_{i,j\in I, i<j}{q_i^2-q_j^2} \ = \ 0 
    \]
    for any $I \subseteq \{1,\dots,l\}$ with $|I| = k$. This shows that $\q \in
    \H_{k-1}(B_n) \subseteq \H_k(D_n)$.

    If $k= \frac{n}{2}$, then $n$ is even and $X_k(\q)$ is cut
    out by $\PowSum_2,\dots, \PowSum_{n-2}$ and possibly $\PowSum_{n}$, since $e_{n}(\x)$ is identically zero on
    $X_k(\q)$. Thus the argument above
    remains valid.

    \textbf{Case $l = n$:} 
    If $k < \frac{n}{2}$, set $f := e_n$.  If $k \ge \lfloor \frac{n}{2}
    \rfloor +1$, set $f := \PowSum_{2k}$. For the special case that $n$ even
    and $k =\frac{n}{2}$, we set $f = e_n$ if $\alpha \neq 0$ and $f =
    \PowSum_{n}$ otherwise.  Let $\r \in X_k(\q)$ be a maximizer of $|f(\x)|$.
    In particular, $r_1\cdots r_n \neq 0$. Up to row and column operations,
    the Jacobian $J = \Jac_\q(\PowSum_2,\PowSum_4,\dots,\PowSum_{2k},f)$ is of
    the form
    \begin{equation}\label{eqn:l_n}
        J \ = \ \begin{pmatrix}
            r_1 & r_2 & \cdots & r_n \\
            r_1^3 & r_2^3 & \cdots & r_n^3 \\
            \vdots & \vdots & & \vdots \\
            r_1^{2k-1} & r_2^{2k-1} & \cdots & r_n^{2k-1} \\
            \widehat{r}_1 r_2\cdots r_n & r_1 \widehat{r}_2\cdots r_n & \cdots
            & r_1 r_2\cdots \widehat{r}_n
        \end{pmatrix},
    \end{equation}
    where $\widehat{r}_i$ is to be omitted from the product. Multiplying the
    $i$-th column by $r_i$ and dividing the last row by $r_1\cdots r_n$, we
    get a Vandermonde matrix of rank $\le k$. Hence $(|r_1|,\dots,|r_k|)$ has
    at most $k$ distinct entries. Since all entries are nonzero, it follows
    that $\r \in \H_k(D_n)$.
\end{proof}

The proof actually gives stronger implications for the $B_n$-case.

\begin{cor}\label{cor:stronger_Bn}
    Let $1\le k\le n-2$. Then every nonempty $B_n$-invariant, $k$-sparse
    variety $X$ meets $\H_k(D_n)$.
\end{cor}
\begin{proof}
    For $X = X_k(\q)$, we
    can assume that $\q =(q_1,\dots,q_l,0,\dots,0)$ with $q_i\neq 0$ for
    $1\le i\le l$ and $l$ maximal. If $l \le k$, then $ \q \in \H_k(D_n)$
    and we are done. So assume $k < l \le n$. We distinguish two cases:
    If $l \le n-1$, let $f = e_{l+1}$ and $\r \in X_k(\q)$ arbitrary.
    If $l = n$, let $f = e_{n}$ and $\r \in X_k(\q)$ a maximizer of $|f|$. The
    corresponding Jacobians~\eqref{eqn:l_n-1} and~\eqref{eqn:l_n} for
    $\PowSum_2,\dots,\PowSum_{2k},f$ at $\r$ yield the claim.
\end{proof}

\section{Real orbit spaces and reflection arrangements}\label{sec:codim1}

The reflection arrangement $\H$ decomposes $V$ into relatively open polyhedral
cones.  The closure $\sigma$ of a full-dimensional cone in this decomposition
serves as a \Defn{fundamental domain}: For every $\p \in V$ the orbit $G\p$
meets $\sigma$ in a unique point; see~\cite[Thm.~1.12]{hum}. On the other
hand, the basic invariants define an \Defn{orbit map} $\pi : V \rightarrow
\R^n$ given by $\pi(\x)=(\pi_1(\x),\ldots,\pi_n(\x))$. The basic invariants
separate orbits, that is, $\pi(\p) = \pi(\q)$ if and only if $\q \in G\p$ for
all $\p,\q \in V$.  The image $\Ospace:=\pi(V)$ is homeomorphic to $V / G$
and, by abuse of terminology, we call $\Ospace$ the \Defn{real orbit space}.
Since $\pi$ is an algebraic map, $\Ospace$ is semialgebraic (with an explicit
description given in~\cite{procesi}). Restricted to $\sigma$ the map
$\pi|_\sigma:\sigma\to \Ospace$ is a homeomorphism, by
\cite[Prop.~0.4]{procesi}.  Moreover, 
\begin{equation}\label{eqn:S_boundary}
    \pi^{-1}(\partial \Ospace) \ = \ \bigcup_{\p \in \partial \sigma} G\p \ =
    \ \H_{n-1}(G),
\end{equation}
where $\partial \Ospace$ denotes the boundary of $\Ospace$ and where the
second equality follows from~\cite[Thm.~1.12]{hum}.  Observe that 
neither the orbit space $\Ospace$ nor the fundamental domain $\sigma$ are
uniquely determined by $G$.

In terms of the orbit map, Conjecture~\ref{conj:main} can be put in a more
general context. For $J \subseteq [n] := \{1,\dots,n\}$, let us write
$\pi_J(\x) = (\pi_i(\x) : i \in J)$. For given $J$, we can ask for the smallest
$0 \le t \le n$ such that $\pi_J(V) \ = \ \pi_J(\H_t)$.

\begin{prop}\label{prop:conj-equiv}
    Let $G$ be an irreducible and essential reflection group. Then
    Conjecture~\ref{conj:main} is true for $G$ if and only if for $J =
    \{1,\dots,k\}$
    \[
        \pi_J(V) \ = \ \pi_J(\H_k).
    \]
\end{prop}
\begin{proof}
    For $\q \in V$, we have $X_k(\q) = \pi_J^{-1}(\pi_J(\q))$. Hence, $X_k(\q)
    \cap \H_k \neq \emptyset$ for $\q \in V$ if and only if there is some $\p
    \in \H_k$ such that $\pi_J(\q) = \pi_J(\p)$.
\end{proof}

A generalization of Theorem~\ref{thm:Sn-degree} to \emph{$J$-sparse} symmetric
polynomials $f \in \R[\pi_i : i \in J]$ was considered in~\cite{Riener14}.
The correspondence given in Proposition~\ref{prop:conj-equiv} also shows that
the dimensions of strata in Conjecture~\ref{conj:main} are best possible.

\begin{prop}\label{prop:lower_bound}
    Let $J \subseteq [n]$ and $0 \le t \le n$ such that $\pi_J(V) \ = \
    \pi_J(\H_t)$. Then $t \ge |J|$.
\end{prop}
\begin{proof}
    The set $\pi_J(V)$ is the projection of the real orbit space $\Ospace$
    onto the coordinates indexed by $J$ and hence is of full dimension $|J|$.
    By invariance of dimension, this implies that $t = \dim \H_t \ge |J|$.
\end{proof}

For the next result recall that, by definition, $G \subset O(V)$ and
hence $\|\x\|^2 = \langle \x, \x \rangle$ is an invariant of $G$.

\begin{lem}\label{lem:S_linefree}
    Let $G$ be a finite reflection group and $\pi_1,\dots,\pi_n$ a choice of
    basic invariants such that $\pi_i(\x) = \|\x\|^2$ for some $i$. Then the
    orbit space $\Ospace = \pi(V)$ is line-free, that is, if $L \subseteq V$
    is an affine subspace such that $L \subseteq \Ospace$, then $L$ is a
    point.
\end{lem}
\begin{proof}
    Since $\pi_i(\x) = \|\x\|^2 \ge 0$ for all $\x \in V$, the linear function
    $\ell(\y) = y_i$ is nonnegative on $\Ospace \subset \R^n$.  Hence, if $L
    \subseteq \Ospace$ is an affine subspace, then $\ell$ is constant on $L$.
    Let $\sigma \subseteq V$ be a fundamental domain for $G$.  Then $L =
    \Ospace \cap L$ is homeomorphic to $\hat{L} := \{ \p \in \sigma : \|\p\|^2
    = c \}$ for some $c \ge 0$. This implies that $L$ is compact which proves
    the claim.
\end{proof}

The following result is a slightly stronger but more technical extension of
Theorem~\ref{thm:all-but-one}.

\begin{thm}\label{thm:codim1}
    Let $G$ be an essential reflection group with a choice of basic invariants
    $\pi_1,\dots,\pi_n$. Let $f \in \R[V]^G$ be an invariant polynomial such
    that $f$ is at most linear in $\pi_k$
    for some $k$. Then $\VarR(f)\neq\emptyset$ if and only if
    $\VarR(f) \cap \H_{n-1} \neq \emptyset$.
\end{thm}

If $f_1,\dots,f_m$ are invariant polynomials that do not depend on $\pi_j$ for
some fixed $j$, then we can apply Theorem~\ref{thm:codim1} to $f = f_1^2 +
\cdots + f_m^2$, which then directly implies Theorem~\ref{thm:all-but-one}.

\newcommand\op{\overline{\p}}%
\begin{proof}
    Without loss of generality, we can assume that $f(0) < 0$.  Since
    $\H_{n-1}$ is path connected, it suffices to show that there is a point
    $\p_+ \in \H_{n-1}(G)$ with $f(\p_+) \ge 0$.  
    
    We can assume that $\pi_1 = \|\x\|^2$. Indeed, since $G$ is
    essential, all basic invariants have degree at least $2$ and $\|\x\|^2$ is
    a linear combination of the degree $2$ basic invariants.  Let $\p \in
    \VarR(f)$ and define $K = \{ \q : \pi_1(\q) = \pi_1(\p)\}$, the sphere
    centered at the origin that contains $\p$. The function $f$ attains its
    maximum over $K$ in a closed set $M \subseteq K$. We claim that
    $M \cap \H_{n-1} \neq \emptyset$. Let $\p_0$ be a point in $M$.

    We may pass to the real orbit space $\Ospace = \pi(V)$ associated to $G$
    and $\pi_1,\dots,\pi_n$ and consider the compact set $\overline{K} :=
    \pi(K) = \{ \y \in \Ospace : y_1 = \pi_1(\p)\}$.  We can write $f =
    F(\pi_1,\dots,\pi_n)$ for some $F \in \R[y_1,\dots,y_n]$.  In this
    setting, our assumption states that $F$ is at most linear in $y_k$.  If
    $\p_0 \in V \setminus \H_{n-1}$, then, by~\eqref{eqn:S_boundary}, $\op_0 :=
    \pi(\p_0)$ is in the interior of $\Ospace$ and hence in the relative
    interior of $\overline{K}$. Let $L = \{ \op_0 + t e_k : t \in \R\}$ be the
    affine line through $\op_0$ in direction $e_k$. Restricted to $L$, the
    polynomial $F$ has degree at most $1$. By Lemma~\ref{lem:S_linefree} and
    our choice of $\overline{K}$, the line $L$ meets $\partial \overline{K}$
    in two points $\op_-,\op_+$ and  $ F(\op_-) \le F(\op_0) \le F(\op_+)$.
    This implies that $\pi^{-1}(\op_+) \subseteq M$ and, since $\partial
    \overline{K} \subseteq \partial \Ospace$, equation~\eqref{eqn:S_boundary}
    shows that $\pi^{-1}(\op_+) \subseteq \H_{n-1}$.
\end{proof}

The assumption in Theorem~\ref{thm:codim1} that $G$ is essential is essential.
For example, let $G = B_n$ act on $V = \R^n \times \R$ by fixing the last
coordinate. A set of basic invariants is given by $\pi_1(\x,x_{n+1}) =
x_{n+1}$ and $\pi_i(\x,x_{n+1}) = \PowSum_{2i-2}(\x)$ for $i = 2,\dots,n+1$.
Pick $\p \in \R^n$ with all coordinates positive and distinct. The variety
\begin{equation}\label{eqn:Bn-counter}
    X \ = \ \{ (\x,x_{n+1}) \in V : \PowSum_{2i}(\x) = \PowSum_{2i}(\p) \text{ 
    for } i = 1,\dots,n \}
\end{equation}
is defined over $\R[\pi_2,\dots,\pi_{n+1}]$, but is a collection of affine
lines that does not meet the reflection arrangement.

We give two further applications of Theorem~\ref{thm:codim1}.

\begin{cor}\label{cor:codim1apps}
    Let $G$ be an essential reflection group and let $J \subset [n]$ with $|J|
    = n-1$. For polynomials $f, f_1,\dots,f_m \in \R[\pi_i : i \in J]$, the
    following hold:
    \begin{enumerate}[\rm (i)]
        \item $\SemiAlg = \{ \p : f_1(\p) \ge 0,\dots, f_m(\p) \ge 0\}$ is
            nonempty if and only if $\SemiAlg \cap \H_{n-1}(G) \neq \emptyset$.
        \item $f(\q) \ge 0$ for all $\q \in \SemiAlg$ if and only if $f(\q)
            \ge 0$ for all $\q \in \SemiAlg \cap \H_{n-1}(G)$.
    \end{enumerate}
\end{cor}
\begin{proof}
    For $\q \in \SemiAlg$, it suffices to prove the claim for 
    \[
        X \ := \  \{ \p \in V : \pi_j(\p) = \pi_j(\q) \text{ for } j \in J\} \
        \subseteq \ S.
    \]
    Claim (i) now follows from Theorem~\ref{thm:codim1}. As for (ii), assume
    that $\q \in \SemiAlg \setminus \H_{n-1}$ and $f(\q) < 0$.  Then the same
    argument applied to $X \cap \{ \p : f(\p)  = f(\q)\}$ finishes the proof.
\end{proof}

If $f \in \R[V]^G$ has degree $\deg(f) < 2 \deg(\pi_n)$, then the algebraic
independence of the basic invariants implies that Theorem~\ref{thm:codim1} can
be applied to proof the following corollary. Under the assumption that $f$ is
homogeneous, the second part of the corollary recovers the main result of
Acevedo and Velasco~\cite{velasco}. 

\begin{cor}
    Let $f \in \R[V]^G$ with $\deg(f)<2d_n(G) = 2\deg(\pi_n)$. Then
    $\VarR(f)\neq\emptyset$ if and only if $\VarR(f) \cap \H_{n-1} \neq
    \emptyset$. In particular, $f \ge 0$ on $V$ if and only if $f \ge 0$ on
    $\H_{n-1}$.
\end{cor}

The bound on the degree is tight: For a point $\p \in V \setminus
\H_{n-1}(G)$, the set of solutions to
\[
    f(\x) \ := \ \sum_{i=1}^n (\pi_{i}(\x) - \pi_{i}(\p))^2 \ = \ 0
\]
is exactly $G\p$, which does not meet $\H_{n-1}(G)$. The defining polynomial
$f(\x)$ is of degree exactly $2\deg(\pi_n)$. Theorem~\ref{thm:codim1} also allows us to prove Theorem~\ref{thm:main} for groups of low rank.

\begin{proof}[Proof of Theorem~\ref{thm:main} for $\rk(G) \le 3$]
    For $k = \rk(G)$, there is nothing to prove. For $k = 1$, we observe that
    $X_1(\p)$ is the sphere through $\p$, which meets the arrangement
    $\H_1(G)$ of lines through the origin. Thus, the only nontrivial case is
    $\rk(G) = 3$ and $k= \rk(G) -1 = 2$. This is covered by Theorem~\ref{thm:all-but-one}.
\end{proof}

\newcommand\Dmin{\delta_{\min}}
\newcommand\Dmax{\delta_{\max}}
Let $G$ be an essential reflection group of rank $\ge 4$. Since $G$ acts on
$V$ by orthogonal transformations, we have that $\pi_1(\x) = \|\x\|^2$ and
$X_k(\p)$ is a subvariety of a sphere centered at the origin.  Since the basic
invariants are homogeneous, we may assume that $\pi_1(\p) = 1$ and hence
$X_k(\p) \subseteq S^{n-1} = \{ \x \in V : \|\x\| = 1 \}$.  To prove
Theorem~\ref{thm:main} for $k=2$ we can proceed as follows. Let $\Dmin$ and
$\Dmax$ be the minimum and maximum of $\pi_2$ over $S^{n-1}$. Then it suffices
to find points $\p_{\min},\p_{\max} \in \H_2(G) \cap S^{n-1}$ with
$\pi_2(\p_{\min}) = \Dmin$ and $\pi_2(\p_{\max}) = \Dmax$. Indeed, since
$\H_2(G)$ is connected (for $\rk(G) \ge 3$), this shows that $\pi_J(V) =
\pi_J(\H_2(G))$ for $J=\{1,2\}$, which, by Proposition~\ref{prop:conj-equiv},
then proves the claim. For the group $F_4$, we can implement this strategy.

\begin{proof}[Proof of Theorem~\ref{thm:main} for $F_4$]
    Since $F_4$ is of rank $4$, we only need to consider the case $k=2$ and
    can use the strategy outlined above. Let $\Dmin$ and $\Dmax$ be the
    minimum and maximum of $\pi_2$ over $S^3$.  An explicit description of
    $\pi_2$ for $F_4$ is
    \[
        \pi_2(\x) \ = \ \sum_{1 \le i < j \le 4} (x_i + x_j)^6 + (x_i -
        x_j)^6;
    \]
    see, for example, Mehta~\cite{Mehta} or~\cite[Table~5]{IKM}. The points
    $\p = (1,0,0,0)$ and $\p' = (\frac{1}{\sqrt{2}}, \frac{1}{\sqrt{2}}, 0,
    0)$ are contained in $\H_1(F_4) \subseteq \H_2(F_4)$ and takes values
    $\pi_2(\p) = 1$ and $\pi_2(\p') = \frac 3 2$. We claim, that these values
    are exactly $\Dmin$ and $\Dmax$, respectively.

    Note that $\pi_2(\x) = g(x_1^2,x_2^2,x_3^2,x_4^2)$ for 
    \[
        g(\y) \ = \ 5 \PowSum_1(\y) \cdot \PowSum_2(\y) - 4 \PowSum_3(\y).
    \]
    Let $\Delta_3 = \{ \x \in \R^4 : x_1,\dots,x_4 \ge 0, x_1 + \cdots + x_4 =
    1 \}$ be the standard $3$-simplex. We have that $\rho(S^3) = \Delta_3$
    where $\rho(x_1,\dots,x_4) := (x_1^2,\dots,x_4^2)$. Hence,
    \[
        \Dmax \ = \ \max \{ g(\p) : \p \in \Delta_3 \} \quad \text{ and }
        \quad \Dmin \ = \ \min \{ g(\p) : \p \in \Delta_3 \}.
    \] 
    Now, $D_4$ is a subgroup of $F_4$ and $\pi_2 \in
    \R[\PowSum_2,\PowSum_4,\PowSum_6]$ and does not depend on $e_4(\x)$. By
    Theorem~\ref{thm:main} for $D_4$, the varieties $S^3 \cap \{ \pi_2(\x) =
    \Dmin\}$ and $S^3 \cap \{ \pi_2(\x) = \Dmax\}$ both meet $\H_3(D_4)$.
    Hence, it suffices to minimize or maximize $g(\x)$ over 
    \[
        \Delta_3 \cap \{ \x \in \R^4 : x_1 = x_2 \}.
    \]
    This leaves us with the (standard) task to maximize and minimize a
    bivariate polynomial $g'(s,t)$ of degree $3$ over a triangle. In the
    plane, the polynomial has $3$ critical points with values $1,
    \frac{11}{9}, \frac{11}{9}$. On the boundary, the extreme values are
    attained at the points given above.
\end{proof}

For the rank-$4$ reflection group $H_4$, the invariant $\pi_2(\x)$ is a
polynomial of degree $12$ in four variables; see, for
example,~\cite[Table~6]{IKM}. Since $(B_1)^4$ is a reflection subgroup of
$F_4$, $\pi_2$ is a polynomial in the squares $x_1^2,\dots,x_4^2$ and,
following the argument in the proof above, we are left with minimizing and
maximizing a degree-$6$ polynomial $g(\x)$ over the simplex $\Delta_3$.
However, finding the critical points is not easy and an extra
computational challenge is the fact that $g(\x)$ is a polynomial with
coefficients in $\Q(\sqrt{5})$.  \textsc{GloptiPoly}~\cite{Glo}
\emph{numerically} computes $\Dmin = -\frac{5}{16}$ and $\Dmax = 1$. These
values are attained at $\p_{\min} = \frac{1}{\sqrt{2}}(1,1,0,0)$ and
$\p_{\max} = (1,0,0,0)$, respectively, and both points lie in $\H_2((B_1)^4)
\subseteq \H_2(H_4)$. This is strong evidence for the validity of
Conjecture~\ref{conj:main} for $H_4$ but, of course, not a rigorous
proof.

\section{Strata of higher codimension}\label{sec:higher_codim}

\newcommand\Hyp{\mathcal{A}}    
\newcommand\Flats{\mathcal{L}}    
\newcommand\Strat{\mathcal{H}}    
\newcommand\OpStrat{\mathcal{H}^\circ}    

\newcommand\halfDeg{\mathfrak{s}}
\newcommand\Gdeg{\deg_G}

We have seen in the previous section that every nonempty $(n-1)$-sparse
variety meets the hyperplane arrangement $\H_{n-1}$. In this section want to
extend this result to $k$-sparse varieties for $k<n-1$. This case is
considerably more difficult but we can make good use of the techniques and
ideas developed in Section~\ref{sec:codim1}.

Let $G$ be an essential finite reflection group acting on $V\cong\R^n$.
Consider a $G$-invariant $k$-sparse variety $X$ with $k<n$. If $X$ is
nonempty, then Theorem~\ref{thm:codim1} yields that for some reflection
hyperplane $H \in \H$ the variety $X':=X\cap H$ is nonempty. An inductive
argument could now replace $G$ by some other reflection subgroup $G' \subseteq
G$ that fixes $H$. If $X'$ remains sparse with respect to $G'$ we can again
apply Theorem~\ref{thm:codim1} to obtain a point
$\p \in \H_{n-2}(G') \subseteq \H_{n-2}(G)$. However, the results obtained using
this strategy are far from optimal. We will briefly illustrate this for
$G=\SymGrp_n$: Let $X$ be a nonempty $k$-sparse $\SymGrp$-invariant variety
for $k<n$. The largest subgroup of $\SymGrp_n$ that fixes a given
reflection hyperplane $H \in \H$ is $G'\cong\SymGrp_{n-2}\times\SymGrp_2$. Hence
Theorem~\ref{thm:codim1} only applies for $X'=X\cap H$ and $G'$ if  
$k=d_k(G)<d_n(G')=n-2$, in other words if the original variety $X$ is
$(k-3)$-sparse. Inductively, this yields that every nonempty $k$-sparse
$\SymGrp_n$-invariant variety meets $\H_l$ where $l=\lfloor
\frac{n+k}{2}\rfloor$.  However, applying the above method to the exceptional
types gives nontrivial bounds.

\begin{prop}\label{prop:fixH}
    Let $6 \le n \le 8$. Then every nonempty $2$-sparse $E_n$-invariant
    variety intersects $\H_{n-2}(E_n)$.
\end{prop} 
\begin{proof}
    We exemplify the argument for the case $n=8$. Let $X$ be a nonempty
    $2$-sparse $E_8$-invariant variety. By Theorem~\ref{thm:codim1} we find a
    point $\p \in X \cap \H_7(E_8)$. The orbit of $\p$ meets every
    hyperplane in $\H(E_8)$ (see \cite[Sect.~2.10]{hum}) and  hence we may assume
    that $\p$ lies on the hyperplane $H=\{\x\in\R^8:x_1=x_2\}$. Consider the
    subgroup $G'\cong D_6\subset E_8$ acting essentially on the coordinates
    $x_3,\dots,x_8$. Since $d_6(D_6)=10>8=d_2(E_8)$, we can apply
    Theorem~\ref{thm:codim1} to finish the proof.
\end{proof}

By restricting the class of invariant polynomials, we obtain better bounds
than those in Proposition~\ref{prop:fixH}. In the following, a point $\p$ is
called \Defn{$\boldsymbol G$-general} if it does not lie on any reflection
hyperplane of $G$, and hence $|G\p| = |G|$.

\begin{definition}
    For a positive integer $d$, let $\strNum(d)$ be the largest number $\ell$
    such that for every $p \in \Strat_{\ell+1}$ there is a reflection subgroup
    $G' \subseteq G$ such that $p$ is $G'$-general and $2d_n(G') > d$.
    Moreover, we define $\SecStrNum(k) := \strNum(2d_k(G))$.  That is,
    $\SecStrNum(k)$ is the largest $0 \le \ell \le d$ such that for every $p
    \in \Strat_{\ell+1}$, there is a reflection subgroup $G' \subseteq G$ such
    that $p$ is $G'$-general and $d_n(G')> d_k(G)$. 
\end{definition}

We call an invariant polynomial $f\in\R[V]^G$ \Defn{$\boldsymbol G$-finite} if
either $V_\R(f)=\emptyset$ or if there is a point $\p \in V_\R(f)$ such that
$f$ has finitely many extreme points restricted to the sphere $K = \{ \q \in V :
\|\q\| = \|\p\|\}$. 

\begin{thm}\label{thm:HalfDeg}\label{thm:DegPrinc}
    Let $f \in \R[V]^G$ be a $G$-finite polynomial and  $X=\VarR(f)$. 
    If $f\in\R[\pi_1,\dots,\pi_k]$, then
    \[
        X \ \neq \ \emptyset \quad \text{ if and only if } \quad
        X \cap \H_{\SecStrNum(k)} \ \neq \ \emptyset.
    \]
    If $d = \deg(f)$, then
    \[
        X \ \neq \ \emptyset \quad \text{ if and only if } \quad
       X \cap \H_{\strNum(d)} \ \neq \ \emptyset.
    \]
\end{thm}
\begin{proof}
    We only give a proof for the second result. The proof of the first is
    analogous. Suppose $\VarR(f) \neq \emptyset$.  We may assume that $f(0)
    \le 0$ and, since $\H_{\strNum(d)}$ is connected, it suffices to show that
    there is some point $\p_+ \in \H_{\strNum(d)}$ with $f(\p_+) \ge 0$.  
    
    By assumption, there is a zero  $\p_0 \in \VarR(f)$ such that $f$ has only
    finitely many extreme points restricted to $K = \{ \q : \|\q\| =
    \|\p_0\|\}$.  Let $\p_+ \in K$ be a point maximizing $f$ over $K$ and
    hence $f(\p_+) \ge f(\p_0) \ge 0$. We claim that $\p_+ \in
    \H_{\strNum(d)}$. Otherwise, there is a reflection subgroup $G' \subset G$
    such that $\p_+ \not \in \H_{n-1}(G')$ and $2d_n(G') > d$.  Let
    $\pi'_1,\dots,\pi'_n$ be a choice of basic invariants of $G'$ and, without
    loss of generality, $\pi'_1(\x) = \|\x\|^2$. Thus, $\op_+ = \pi'(\p_+)$ is
    in the interior of $\Ospace = \pi'(V)$. We can write $f =
    F(\pi'_1,\dots,\pi_n)$ for some $F \in \R[y_1,\dots,y_n]$. On the level of
    orbit spaces, our assumption states that restricted to $\overline{K} =
    \pi(L) = \{ \y \in \Ospace : y_1 = \pi_1(\p_+)\}$, the polynomial $F$ has
    only finitely many extreme points.  However, $F$ is linear in $y_n$ and
    thus $\op_+$ is a maximum only if $\op_+ \in \partial \overline{K}
    \subseteq \partial \Ospace$. This is a contradiction.
\end{proof}

\section{Bounds from parabolic subgroups}\label{sec:StrNum}

The numbers $\strNum(d)$ and $\SecStrNum(k)$ defined in
Section~\ref{sec:higher_codim} are difficult to compute in general.  In this
section, we compute upper bounds on these numbers coming from parabolic
subgroups.  Let $G$ be a finite irreducible reflection group. A fundamental
domain $\sigma \subset V$, as defined in Section~\ref{sec:codim1}, is a
simplicial cone of dimension $n = \dim V$. Let $H_1,\dots, H_n \in \H(G)$ be
the reflection hyperplanes that are facet-defining for $\sigma$ and let
$\Delta = \{s_1,\dots, s_n\} \subset G$ be the corresponding reflections. For
$i \neq$, we denote by $m(i,j)$ the order of the cyclic group generated by
$s_is_j$. The \Defn{Dynkin diagram} $D$ of $G$ is the labelled graph with
vertex set $\{1,\dots,n\}$ and edges $ij$ whenever $m(i,j) \ge 3$ and edge
labelling $m(i,j)$; see~\cite[Sec.~2.1]{hum} for details.  A subgroup of $G$
is \Defn{parabolic} if it is conjugate to a subgroup generated by a subset of
the reflections in $\Delta$; cf.~\cite[Sec.~1.10]{hum}.

\begin{lem}\label{lem:leaf}
    Fix a finite irreducible reflection group $G$ with Dynkin diagram $D$.
    Let $D' \subset D$ be a subdiagram obtained by removing a node from $D$
    and let $H \in \H(G)$ be a reflection hyperplane. Then there is a
    parabolic subgroup $W \subset G$ with Dynkin diagram $D'$ and $H$ is not
    a reflection hyperplane of $W$. 
\end{lem}
\begin{proof}
    Let $W$ be a parabolic subgroup with Dynkin diagram $D'$.  Since every
    parabolic subgroup with Dynkin diagram $D'$ is conjugate to $W$, it
    suffices to show that for every $H\in\H(G)$ there is a $g \in G$ such that $g H \not\in \H(W)$. 

Unless $G$ is of type $F_4$, $B_n$ or $I_2(2m)$ for $m > 1$, the hyperplanes in $\H(G)$ form a single $G$-orbit (see~\cite[Sec.~2.9, Sec.~2.10]{hum}). Hence, the claim follows, since $\H(W)\subsetneq\H(G)$.
 For $F_4$, the possible
    proper parabolic subgroups are $B_3$ and $A_1 \times A_2$ and the result
    follows by inspection.  For $I_2(2m)$, there are two orbits with each $2m
    \ge 4$ roots whereas the only nontrivial proper parabolic subgroup is
    $A_1$. For $B_n$, this follows from counting the number of elements in
    each of the two orbits.
\end{proof}

The lemma yields the following result about finite reflection groups that
might be interesting in its own right.

\begin{prop}\label{subdiagram}
    Let $G$ be a finite irreducible reflection group with Dynkin diagram $D$
    acting on a real vector space $V$. For $k \ge 1$, let $\p \in \H_k
    \setminus \H_{k-1}$ and $D' \subset D$ be a connected subdiagram on $k$
    nodes. Then there is a parabolic subgroup $W \subset G$ with Dynkin
    diagram $D'$ such that $\p$ is $W$-general.
\end{prop}
\begin{proof}
    We argue by induction on $s = \dim V - k$.  For $s = 0$, $\p \in V
    \setminus \H_{\dim V -1}$ and $\p$ is by definition $G$-general. Otherwise, let $D_1
    \subset D$ be a subdiagram obtained by removing a leaf such that $D'
    \subseteq D_1$ and let $H_1$ be a reflection hyperplane of $G$ containing
    $\p$.  We may use
    Lemma~\ref{lem:leaf} to obtain a parabolic subgroup $W_1$ with Dynkin
    diagram $D_1$ and not containing $H_1$ as a reflection hyperplane. In
    particular, $\p$ is contained in precisely $s-1$ linearly independent
    reflection hyperplanes of $W_1$. By induction, there is a parabolic
    subgroup $W \subseteq W_1$ with Dynkin diagram $D'$ for which $\p$ is
    $W$-general. In particular, $W$ is a parabolic subgroup of $G$ which
    concludes the proof.
\end{proof}

For $I \subseteq \Delta$, let $W_I$ be the parabolic subgroup generated by the
reflections $I$. We define
\[
    \parNum(d) \ := \ \min\{ |I|-1 : I\subseteq \Delta, 2d_n(W_I) > d \},
\]
and, analogously, we define $\SecParNum(k) := \parNum(2d_k(G))$. Since $G$ acts transitively on the set of fundamental domains, these definitions do not depend on the choice of $\sigma$. Proposition~\ref{subdiagram} implies the following bound on $\strNum$.

\begin{cor}\label{cor:parabolic_bound}
    $ \strNum(d) \le \parNum(d),  $
    for all $d \ge 0$.
\end{cor}

The clear advantage is a simple way to compute upper bounds on $\strNum(d)$
from the knowledge of parabolic subgroups of reflection groups;
cf.~\cite{hum}. The explicit values are given in Table~\ref{tab:parNum}.
However, not every reflection subgroup is parabolic (e.g. $D_n\subseteq B_n$,
$I_2(m) \subseteq I_2(2m)$).  Nevertheless, we conjecture that $\strNum(d)$ is
attained at a parabolic subgroup.

\begin{conj}
    For any finite reflection group $G$,
    $
        \strNum(d) \ = \ \parNum(d)
    $
    for all $d$.
\end{conj}

\newcommand\dd{\text{ --- }}%
\begin{table}
    \newcommand\mc[1]{\multicolumn{2}{c|}{#1}}

\[
\scriptsize
\begin{array}{c|r@{\dd}l|c|c|c||r@{\dd}l|c}
    G & \multicolumn{2}{c|}{d} & \parNum(d) & W&d_n(W) & \mc{k} & \SecParNum(k)\\
        \hline
        \hline
        A_{n-1} / \mathfrak{S}_n & 0 & 2n-1 & \lfloor d/2\rfloor& A_{\lfloor d/2\rfloor}&\lfloor d/2\rfloor+1 & 0 &  n-1 & k\\
        \hline
        B_n & 0 & 4n-1 & \lfloor d/4\rfloor& B_{\lfloor d/4\rfloor+1} & 2(\lfloor d/4\rfloor+1)& 0 & n-1 & k\\
        \hline
         
        D_n& 0 & 4n-5 &\lfloor d/4\rfloor+1  & D_{\lfloor d/4\rfloor+2} &2(\lfloor d/4\rfloor+1) & 0 & \lfloor\frac{n}{2}\rfloor & k+1\\
        &\multicolumn{1}{c}{}&&&&& \lfloor\frac{n}{2}\rfloor+1 & n & k\\
        \hline
        I_2(m) & 1 & 2m-1 & 1&I_2(m) & m& \mc{1} & 1\\

        \hline
        E_6 & 1 & 5 & 1&A_2&3& \mc{1} & 1\\
            & 6 & 7 & 2&A_3&4& \mc{2} & 3\\
            & 8 & 11 & 3&D_4&6& \mc{3} & 4\\
            & 12 & 15 & 4&D_5&8& \mc{4} & 5\\
            & 16 & 23 & 5&E_6&12& \mc{5} & 5\\
        \hline
        E_7 & 1 & 5 & 1&A_2&3& \mc{1} & 1\\
            & 6 & 7 & 2&A_3&4& \mc{2} & 4\\
            & 8 & 11 & 3&D_4&6& \mc{3} & 5\\
            & 12 & 15 & 4&D_5&8& \mc{4} & 5\\
            & 16 & 23 & 5&E_6&12& \mc{5} & 6\\
            & 24 & 35 & 6&E_7&18& \mc{6} & 6\\
\hline
    E_8 & 1 & 5 & 1&A_2&3& \mc{1} & 1\\
        & 6 & 7 & 2&A_3&4& \mc{2} & 5\\
        & 8 & 11 & 3&D_4&6& \mc{3} & 6\\
        & 12 & 15 & 4&D_5&8& \mc{4} & 6\\
        & 16 & 23 & 5&E_6&12& \mc{5} & 7\\
        & 24 & 35 & 6&E_7&18& \mc{6} & 7\\
        & 36 & 59 & 7&E_8&30& \mc{7} & 7\\
    \hline
    F_4 & 1 & 7 & 1&B_2&4& \mc{1} & 1\\
        & 8 & 11 & 2&B_3&6& \mc{2} & 3\\
        & 12 & 23 & 3&F_4&12& \mc{3} & 3\\
    \hline
    H_3 & 1 & 9 & 1&I_2(5)&5& \mc{1} & 1\\
        & 10 & 19 & 2&H_3&10& \mc{2} & 2\\
    \hline
    H_4 & 1 & 9 & 1&I_2(m)&5& \mc{1} & 1\\
        & 10 & 19 & 2&H_3&10& \mc{2} & 3\\
        & 20 & 59 & 3&H_4&30& \mc{3} & 3\\
\end{array}
\quad \quad \quad
    \]
    \caption{Computation for $\parNum(d)$ and $\SecParNum(k)$.
    $a\dd b$ refers to the range $a,a+1,...,b$. The column $W$ gives the
    parabolic subgroup that attains $\parNum$.}\label{tab:parNum}
\end{table}

\section{Adjoint representations of Lie groups}
\label{sec:lie}

\newcommand\g{\mathfrak{g}}
In this last section, we extend some of our results to polynomials invariant
under the action of a Lie group. More precisely, we consider the case of a
real simple Lie group $G$ with the adjoint action on its Lie algebra $\g$.
\newcommand\GL{\mathrm{GL}}%
\newcommand\SL{\mathrm{SL}}%
\newcommand\gl{\mathfrak{gl}}%
\newcommand\sln{\mathfrak{sl}}%
\newcommand\tr{\mathrm{tr}}%
\newcommand\X{\mathbf{X}}%
We illustrate our results for the case $G = \SL_n$.  Its Lie algebra $\sln_n$
is the vector space of real $n$-by-$n$ matrices of trace $0$.  The adjoint
action of $\SL_n$ on $\sln_n$ is by conjugation: $g \in \SL_n$ acts on $A \in
\sln_n$ by $g\cdot A := gAg^{-1}$. The following description of its ring of
invariants is well-known. We briefly recall the standard proof which
immediately suggests a connection to our treatment of reflection groups;
cf.~\cite[Ch.~12.5.3]{GW}.
\begin{thm}\label{thm:GL}
    For $n \ge 1$, $ \R[\sln_n]^G  =  \R[\PowSum_2,\dots,\PowSum_n],
    $
    where $\PowSum_k(A) = \tr(A^k)$ for $k=2,\dots,n$. Moreover,
    $\PowSum_2,\dots,\PowSum_n$ are
    algebraically independent.
\end{thm}
\begin{proof}
    We write $D \subset \sln_n$ for the set of diagonalizable matrices and we
    denote by $\lambda(A) = (\lambda_1,\dots,\lambda_n)$ the eigenvalues of $A
    \in D$. Then for any $A \in D$
    \[
        \PowSum_k(A) \ = \  \PowSum_k(\lambda(A)) \ = \ \lambda_1(A)^k + \lambda_2(A)^k +
        \cdots + \lambda_n(A)^k
    \]
    and $\PowSum_2,\dots,\PowSum_n$ are simply the power sums restricted
    to the linear subspace $\Delta \subset D$ of diagonal matrices. This shows
    that $\PowSum_2,\dots,\PowSum_n$ are algebraically independent. Now for a polynomial
    $f(\X) \in \R[\sln_n]$ invariant under the action of $\SL_n$, the
    restriction to $\Delta \cong \R^{n-1}$ is a polynomial $f(\x)$ that is
    invariant under $A_{n-1}$. Hence $f(\x) = F(\PowSum_2(\x),\dots,\PowSum_n(\x))$ for
    some $F \in \R[y_2,\dots,y_n]$. The polynomial $\tilde f(\X) =
    F(\PowSum_2(\X),\dots,\PowSum_n(\X))$ is invariant under $\SL_n$ and agrees with $f$
    on $D$. Since $D$ contains a nonempty open set, $f = \tilde f$ as
    required.
\end{proof}

\newcommand\SO{\mathrm{SO}}%
\newcommand\so{\mathfrak{so}}%
\newcommand\pf{\mathrm{pf}}%
\newcommand\torus{\mathfrak{t}}%
\newcommand\D{\mathcal{D}}%
In general, let $T \subseteq G$ be a maximal torus with Lie algebra $\torus
\subseteq \g$. If $N \subseteq G$ is the normalizer of $T$ in $G$, then $W =
N/T$ is a reflection group, the \Defn{Weyl group}, that acts on $\torus$. By
Chevalley's Restriction Theorem (see~\cite[Lem.~7]{HC}), the restriction
$\R[\g] \to \R[\torus]$ extends to an isomorphism of invariant rings.  This
yields that $\R[\g]^G$ is generated by homogeneous and algebraically
independent polynomials $\pi_1,\dots,\pi_m$ whose restriction to $\torus$
give a set of basic invariants for $W$.  We call a $G$-invariant variety
$X\subseteq\g$ $k$-sparse if $X = \VarR(f_1,\dots,f_m)$ for some
$f_1,\dots,f_m \in \R[\pi_1,\dots,\pi_k]$. The \Defn{discriminant locus} $\D
\subset \g$ of $G$ is the Zariski closure of the orbit of the reflection
arrangement $\H(W) \subset \torus$ under $G$. By the result of
Steinberg~\cite{stein} and the Restriction theorem, this is a real
$G$-invariant hypersurface given by the vanishing of a single polynomial,
called the \Defn{discriminant} of $G$. Since $G \cdot \torus$ is dense in
$\g$, $\D$ is the closure of the set of points with nontrivial stabilizer;
see~\cite[Ch.~6]{Hel}.  This yields a stratification of $\g$ by defining
$\D_i$ to be the closure of the orbit of $\H_i$, which corresponds to the
points for which the discriminant vanishes up to order $n-i$. Hence, the
results from the previous sections generalize to Lie groups.  

For the special orthogonal group $\SO_{n}$, its Lie algebra $\so_n \subset
\sln_n$ is the vector space of skew-symmetric $n$-by-$n$ matrices on which
$\SO_{n}$ acts by conjugation. If $n = 2k+1$, then the corresponding Weyl
group is $B_k$ and $D_k$ if $n=2k$. Hence $\R[\so_{2k+1}]^{\SO_{2k+1}}$ is
generated by $\PowSum_2(\X),\PowSum_4(\X),\dots,\PowSum_{2k}(\X)$. For $n=2k$,
a minimal generating set is given by
$\PowSum_2(\X), \PowSum_4(\X)$, ..., $\PowSum_{2n-2}(\X)$ and the Pfaffian
$\pf(\X) = \sqrt{\det \X }$.
Theorem \ref{thm:main}
yields the following.

\begin{thm}\label{thm:Lie_codim1}
    Let $G \in \{ \SL_k, \SO_k : k \in \Z_{\ge 1} \}$ and let $X\subseteq\g$ be $G$-invariant and $k$-sparse. If $X$ is nonempty it intersects $\D_k$.
    \end{thm}

For suitable Lie groups $G$, Theorem~\ref{thm:Lie_codim1} gives a first relation
between real varieties invariant under the action of $G$ and the discriminant
locus and Conjecture~\ref{conj:main} is reasonable for this setting.
It would be very interesting to explore this connection further.

\bibliographystyle{siam}
\bibliography{DegreePrinciples}

\end{document}